\documentclass[10pt,a4paper]{article}
\usepackage[dvips]{color}
\usepackage{indentfirst,latexsym,bm}
\usepackage{graphicx}
\usepackage{amsmath,amsthm}
\usepackage{amssymb}
\usepackage{multicol}
\usepackage{ulem}
\usepackage{amsfonts}
\usepackage{mathrsfs}
 \newtheorem{thm}{Theorem}[section]
 \newtheorem{cor}[thm]{Corollary}
 \newtheorem{lem}[thm]{Lemma}
 \newtheorem{prop}[thm]{Proposition}
 
 \theoremstyle{definition}

 \newtheorem{rem}{Remark}

 \DeclareMathAlphabet{\mathsfsl}{OT1}{cmss}{m}{sl}
 
 \newcommand{\Rnum}{\mathbb{R}}
 
 \newcommand{\Znum}{\mathbb{Z}}
 
 \newcommand{\mi}{\mathrm{i}}
 
  \newcommand{\diag}{\mathrm{diag}}
 \newcommand{\tensor}[1]{\mathsf{#1}}
 \newcommand{\abs}[1]{\left\vert#1\right\vert}
 \newcommand{\set}[1]{\left\{#1\right\}}

\title{On the Imbedding Problem for Three-state Time Homogeneous Markov Chains with Coinciding Negative Eigenvalues}

\author{\rm\small
\noindent CHEN Yong,\quad CHEN Jianmin\\
\noindent \footnotesize School of Mathematics and Computing Science, Hunan
University of Science and Technology,\\
\noindent \footnotesize Xiangtan, Hunan, 411201,
P.R.China. chenyong77@gmail.com}
\date{}
\begin{document}
\maketitle
\maketitle \noindent {\bf Abstract } \\
For an indecomposable $3\times 3$ stochastic matrix (i.e., 1-step transition probability matrix) with coinciding negative eigenvalues, a new necessary and sufficient condition of the imbedding problem for time homogeneous Markov chains is shown by means of an alternate parameterization of the transition rate matrix (i.e., intensity matrix, infinitesimal generator), which avoids calculating matrix logarithm or matrix square root.
In addition, an implicit description of the imbedding problem for the $3\times 3$ stochastic matrix in \cite{jhson} is pointed out.

\noindent { \bf keywords: } imbedding problem; three-state time homogeneous Markov chains; negative eigenvalues

\maketitle
\section{Introduction}
The imbedding problem for finite Markov chains has a long history and was first posed by Elfving \cite{Elfv}, which has applications to population movements
in social science \cite{Sil}, credit ratings in mathematical finance \cite{irw}, and statistical inference for Markov processes \cite{BS,Met}. For a review of the imbedding problem, the reader can refer to \cite{cart}.

According to Kingman \cite{irw,King}, the imbedding problem is completely solved for the case of $2\times 2$ matrices by D. G. Kendall, who proved
that a $2\times 2$ transition probability matrix is compatible with a continuous Markov process if and only if
the sum of the two diagonal entries is larger than $1$.

The explicit description of the imbedding problem for the $3\times 3$ stochastic matrix with distinct eigenvalues
or with coinciding positive eigenvalues is shown by Johansen in \cite{jhson}.
When the common eigenvalue is negative, P. Carette  provides several necessary and sufficient conditions to characterize the imbedded stochastic matrix \cite[Theorem 3.3, Theorem 3.6]{cart}. But in contrast to the above conclusions of Kingman or Johansen, it is not clear-cut.

Let $\tensor{I}$ be the identity matrix. Let $\tensor{P},\,\tensor{P}^{\infty}$ be the $3\times 3$ stochastic matrix and its limiting probability matrix respectively. Theorem 3.3 ( resp. Theorem 3.6,) of \cite{cart} needs to calculate square root of the matrix $\tensor{P}^{\infty}-\tensor{I}$ (resp. $\tensor{P}$), and then to test whether each of the off-diagonal elements satisfies an inequality. But the matrix square root is many-valued, just like the matrix logarithm \cite[p22]{Sil}.

In the present paper, a new necessary and sufficient condition is shown, which overcomes the difficulty of uncountably many versions of logarithm or square root (Theorem~\ref{main}). We chiefly rely on an alternate parameterization of the transition rate matrix (Eq.(\ref{parameter})) for the proof. At the same time, for a fixed $\tensor{P}^{\infty}$, the exact lower bound of the eigenvalue of $\tensor{P}$ that makes
$\tensor{P}$ embeddable is given (Remark~\ref{rem})\footnote{The existence of the lower bound is shown in \cite[Theorem~3.7]{cart}}.

In the more general context of time-inhomogeneous Markov chains, the imbedding problem is dealt with by some authors \cite{fryS,fry1,fry2,fug,jhr}. But we only focus on the time-homogeneous Markov chains here.
\section{The imbedding problem for 3-order transition Matrix with coinciding negative eigenvalues}
The transition matrix $\tensor{P}$ is called embeddable if there is a transition rate matrix $\tensor{Q}$ for
which $\tensor{P}=e^{\tensor{Q}}$.

Let $\tensor{P}=(p_{ij})$ be a $3\times 3$ transition probability matrix. Suppose $\tensor{P}$ is indecomposable, i.e., its state space does not contain two disjoint closed sets\cite[P17]{chung}. Let the
unique stationary probability distribution be $\mu'=(\mu_1,\mu_2,\mu_3)$  with $\mu_1+\mu_2+\mu_3=1$. Let $\vec{e}=(1,1,1)'$

\begin{lem}\label{lem0}
  Suppose $\tensor{P}$ is indecomposable. If $\tensor{Q}$ is a transition rate matrix such that $\tensor{P}=e^{h\tensor{Q}}$, $h>0$, then $\tensor{Q}$ has $\mu'\tensor{Q}=0$.
\end{lem}
\begin{proof}
If a distribution $\nu$ has $\nu'\tensor{Q}=0$ (i.e., the left eigenvector with eigenvalue $0$), then $\tensor{P}=e^{h\tensor{Q}}=\sum^{\infty}_{n=0}\frac{1}{n!}h^n\tensor{Q}^n$ implies that $\nu'\tensor{P}=\nu'$.
Since $\tensor{P}$ is indecomposable, one has that $\nu=\mu$.
\end{proof}
\begin{lem}\label{lem1}
Suppose the transition matrix $\tensor{P}$ is embeddable with eigenvalues $\set{1,\,\lambda,\,\lambda}$,
$\lambda<0$. Then $\tensor{P}$ is diagonalizable and satisfies
\begin{equation}\label{PP}
  \tensor{P}=\tensor{P}^{\infty}+\lambda(\tensor{I}-\tensor{P}^{\infty}),
\end{equation}
where $\tensor{P}^{\infty}=\vec{e}\mu'$ is the limiting probability matrix, $\tensor{I}$ is the identity matrix.\footnote{This conclusion is asserted in
\cite{jhson}.}
\end{lem}
\begin{proof}
If $\tensor{Q}$ is a transition rate matrix such that $\tensor{P}=e^{\tensor{Q}}$, then $\tensor{Q}$ has a pair of
conjugate complex eigenvalues and is diagonalizable. Since $\tensor{P}=e^{\tensor{Q}}$, $\tensor{P}$ is diagonalizable too.

The eigenvalues of $\tensor{P}$ are $\set{1,\,\lambda,\,\lambda}$, thus the rank of the matrix $\lambda\tensor{I}-\tensor{P}$ is $1$.
Since $(\lambda\tensor{I}-\tensor{P})\vec{e}=(\lambda-1)\vec{e}$, the three rows of $\lambda\tensor{I}-\tensor{P}$ are equal.
Then
\begin{equation}\label{mat}
\tensor{P}  =
      \left[ \begin{array}{lll}
        1-(x+y)&x&y\\
           z&1-(y+z)&y\\
           z&x&1-(z+x)
       \end{array}
\right],
\end{equation}
and $1-\lambda=x+y+z$. Note that $\mu'=(z,\,x,\,y)/(x+y+z)$, this ends the proof.
\end{proof}
\begin{cor}\label{cor2}
  Suppose $\tensor{P}$ satisfies the condition of Lemma~\ref{lem1}, then the
stationary probability distribution is positive and all elements of $\,\tensor{P}$ are positive.
\end{cor}
\begin{proof}
  Since $\tensor{P}=\lambda\tensor{I}+(1-\lambda)\tensor{P}^{\infty}$, one has that $\lambda+(1-\lambda)\mu_i\geqslant 0,\,i=1,2,3$. Then
  $\mu_i\geqslant \frac{-\lambda}{1-\lambda}>0$, and the off-diagonal elements satisfy $p_{ij}=(1-\lambda)\mu_j>0,\,i\neq j$.
In addition, it was shown by Goodman that each of the diagonal elements of an
embeddable matrix dominates the determinant, and that this determinant is
positive: $p_{ii}\geqslant \det \tensor{P}> 0$ \cite{jhr,gdm}.
\end{proof}
The conclusions of Lemma~\ref{lem1} and Corollary~\ref{cor2} also appear in \cite{cart}. By Lemma~\ref{lem1}, $\tensor{P}$ is completely determined by its stationary distribution and the coinciding eigenvalues.

\begin{prop}\label{prop3}
Suppose that $\tensor{P}$ satisfies Eq.(\ref{PP}). $\tensor{P}$ can be imbedded if and only if there exists a transition
rate matrix $\tensor{Q}$ such that it has $\mu'\tensor{Q}=0$ and eigenvalues $\theta$ and $\bar{\theta}$ and it holds that $e^{\theta h}=\lambda$ for some $h\in \Rnum^+$.
\end{prop}
The proof of Proposition~\ref{prop3} is presented in Section~\ref{prf}. Here $\theta$ is a complex eigenvalue $-p+\mi q$
with $\frac{q}{p}=\frac{(2k+1)\pi}{-\log\abs{\lambda}}$, $k\in \Znum^+$.
\begin{rem}
By the Runnenberg condition in \cite{Sil,Run}, the complex eigenvalue $-p+\mi q$ of the transition
rate matrix $\tensor{Q}$ satisfies that
\begin{equation*}
  \frac{\abs{q}}{p} \leqslant\frac{1}{\sqrt{3}},
\end{equation*}
which the reader can also refer to \cite{cy} for detail. Then
$$\abs{\lambda}\leqslant e^{-\sqrt{3} \pi}\doteq  0.0043,$$
and $\tensor{P}$ is almost equal to its limiting probability matrix by Eq.(\ref{PP}).
\end{rem}

For a transition rate matrix
\begin{equation}
\tensor{Q}= \left[
\begin{array}{lll}
     -a_2-a_3&a_2 &a_3\\
     b_1&-b_1-b_3&b_3\\
     c_1&c_2 &-c_1-c_2
\end{array}
\right ],
\end{equation}
suppose that it has $\mu'\tensor{Q}=0$, that is to say,
\begin{equation}\label{flux}
  \mu_1a_2-\mu_2b_1=\mu_2 b_3 - \mu_3 c_2 = \mu_3 c_1 - \mu_1 a_3.
\end{equation}
Let
\begin{equation}\label{gamma}
  \hspace{-9mm} \nu=\frac{\mu_1a_2-\mu_2b_1}{2},\,
\gamma=\frac{\mu_1a_2+\mu_2b_1}{2},\,
\delta=\frac{\mu_2b_3+\mu_3c_2}{2},\,
\kappa=\frac{\mu_3c_1+\mu_1a_3}{2} .
\end{equation}

We propose an alternate parameterization of the transition rate matrix as
\begin{equation}\label{parameter}
\tensor{Q}=\left[
\begin{array}{lll}
                  -\frac{\kappa+\gamma}{\mu_1}&\frac{\gamma+\nu}{\mu_1}&
                                      \frac{\kappa-\nu}{\mu_1}\\
                  \frac{\gamma-\nu}{\mu_2}&-\frac{\gamma+\delta}{\mu_2}&
                                      \frac{\delta+\nu}{\mu_2}\\
                 \frac{\kappa+\nu}{\mu_3}&\frac{\delta-\nu}{\mu_3}&
                                      -\frac{\delta+\kappa}{\mu_3}
\end{array}
\right ],
\end{equation}
where
\begin{equation}\label{vkg}
  \kappa,\,\gamma,\,\delta\geqslant 0,\quad  \kappa+\gamma,\,\gamma+\delta,\,\delta+\kappa>0, \quad\text{  and} \quad \abs{\nu}\leqslant\kappa,\,\gamma,\,\delta.
\end{equation}
The re-parameterization is one-to-one.
Then the transition
rate matrix $\tensor{Q}$ with $\mu'\tensor{Q}=0$ must satisfy
Eq.(\ref{parameter}).
\footnote{This re-parameterization of the transition rate matrix
appears in \cite{cy}, and appears in \cite{qianq,jdq,Kalpa,McM} implicitly, which is named the cycle decomposition by some authors.}

Let the eigen-equation of
$\tensor{Q}$ be $\lambda(\lambda^2 + \alpha \lambda +\beta)=0$.
Then we have that
\begin{eqnarray}
\alpha &=&\frac{\kappa+\gamma}{\mu_1}+
\frac{\gamma+\delta}{\mu_2}+\frac{\delta+\kappa}{\mu_3},\label{alpha}\\
\beta &=& \frac{ \kappa\gamma + \gamma\delta + \delta\kappa+\nu^2}{\mu_1\mu_2\mu_3 }.\label{beta}
\end{eqnarray}
and the eigenvalue is
\begin{equation}
  \theta =-\frac{\alpha}{2}+\mi \sqrt{\beta-\alpha^2/4}\,.
\end{equation}
Let the ratio
between the imaginary ($q$) and real ($-p$) parts of the
nonzero eigenvalues be
\begin{eqnarray}\label{H-func}
  H(\kappa,\gamma,\delta,\nu)&\triangleq & \frac{\abs{q}}{p} \nonumber \\
  &=&\sqrt{\frac{4\beta-\alpha^2}{\alpha^2}}
  =\sqrt{4\frac{\beta}{\alpha^2}-1}\nonumber \\
  &=&\sqrt{\frac{4}{\mu_1\mu_2\mu_3}\frac{\kappa\gamma + \gamma\delta + \delta\kappa+\nu^2}{(\frac{\kappa+\gamma}{\mu_1}+
\frac{\gamma+\delta}{\mu_2}+\frac{\delta+\kappa}{\mu_3})^2}-1}.
\end{eqnarray}
\begin{rem}
Since there may be many different transition rate matrices such that
$H(\kappa,\gamma,\delta,\nu)=\frac{(2k+1)\pi}{-\log\abs{\lambda}},\,k\in\Znum^+$,
the solution of $\tensor{P}=e^{\tensor{Q}}$ is not unique.
\end{rem}
\begin{rem}\label{prop4}
For the given $\mu'=(\mu_1,\,\mu_2,\,\mu_3)$,
if $\nu=0$ and $\kappa:\gamma:\delta=\frac{1}{\mu_2}:\frac{1}{\mu_3}:\frac{1}{\mu_1}$,
then $H(\kappa,\gamma,\delta,\nu)=0$. Note that when
\begin{equation}
\kappa,\,\gamma,\,\delta\geqslant 0,\quad  \kappa+\gamma,\,\gamma+\delta,\,\delta+\kappa>0, \quad\text{  and} \quad \abs{\nu}\leqslant\kappa,\,\gamma,\,\delta,
\end{equation}
$H(\kappa,\gamma,\delta,\nu)$ is a continuous real function of $(\kappa,\gamma,\delta,\nu)$.\footnote{It is needed that the term inside the $\sqrt{\cdot}$ of Eq.(\ref{H-func}) is positive.} Therefore, $H(\kappa,\gamma,\delta,\nu)$ takes the value $\frac{\pi}{-\log\abs{\lambda}}$ if the maximum of $H(\kappa,\gamma,\delta,\nu)$ is greater than or equal to $\frac{\pi}{-\log\abs{\lambda}}$.
\end{rem}

Therefore, an optimization problem is formulated, and the next Proposition solves it.
We say that three positive numbers $a,\,b,\,c$ satisfy the triangle inequality if $a + b > c,\,b + c > a,\,a + c >b$.
Let
\begin{equation}
  m=\min\set{\mu_1,\,\mu_2,\,\mu_3}.
\end{equation}
\begin{prop}\label{lem5}
Suppose that
\begin{equation}
  F(x_1, x_2, x_3, \mu_1, \mu_2, \mu_3)=\frac{x_1x_2+x_2x_3+x_3x_1+(\min\set{x_1,x_2, x_3})^2}{(\frac{x_1+x_2}{\mu_1}+
\frac{x_2+x_3}{\mu_2}+\frac{x_3+x_1}{\mu_3})^2},
\end{equation}
with $x_1, x_2, x_3\geqslant0$ and $x_1+x_2,\,x_2+x_3,\,x_3+x_1 >0$.
Then the maximum of $F(x_1, x_2, x_3, \mu_1, \mu_2, \mu_3)$ is
\begin{equation}\label{maximum}
  \left\{
      \begin{array}{ll}
      \frac{1}{(\frac{1}{\mu_1}+\frac{1}{\mu_2}+\frac{1}{\mu_3})^2}, &\mbox{\, when\, } \frac{1}{\mu_1},\frac{1}{\mu_2},\frac{1}{\mu_3} \mbox{\, satisfy the triangle inequality, }\\
      \frac{\mu_1\mu_2\mu_3}{4(1-m )}, &\mbox{otherwise}.
      \end{array}
\right.
\end{equation}
The maximum is attained at the points
\begin{equation}\label{maximumpoint}
  \left\{
      \begin{array}{ll}
      x_1= x_2= x_3, &\mbox{\, when\, } \frac{1}{\mu_1},\frac{1}{\mu_2},\frac{1}{\mu_3} \mbox{\,satisfy the triangle inequality, }\\
      x_1= x_2=\frac{\frac{1}{\mu_2}+\frac{1}{\mu_3}}{\frac{2}{\mu_1}-\frac{1}{\mu_2}-\frac{1}{\mu_3}} x_3, &\mbox{\, when\,  } \frac{1}{\mu_1}\geqslant\frac{1}{\mu_2}+\frac{1}{\mu_3} ,\\
      x_2= x_3=\frac{\frac{1}{\mu_3}+\frac{1}{\mu_1}}{\frac{2}{\mu_2}-\frac{1}{\mu_3}-\frac{1}{\mu_1}} x_1, &\mbox{\, when\,  } \frac{1}{\mu_2}\geqslant\frac{1}{\mu_3}+\frac{1}{\mu_1},\\
      x_3= x_1=\frac{\frac{1}{\mu_1}+\frac{1}{\mu_2}}{\frac{2}{\mu_3}-\frac{1}{\mu_1}-\frac{1}{\mu_2}} x_2, &\mbox{\, when\,  } \frac{1}{\mu_3}\geqslant\frac{1}{\mu_1}+\frac{1}{\mu_2}.
      \end{array}
\right.
\end{equation}
\end{prop}
Proof of Proposition~\ref{lem5} is presented in Section~\ref{prf}.

\begin{thm}\label{main}
 Suppose that $\tensor{P}$ is a $3\times 3$ stochastic matrix
with eigenvalues $\set{1,\lambda,\lambda},\,\lambda<0,$ and that $\tensor{P}$ satisfies Eq.(\ref{PP}).
Then $\tensor{P}$ can be imbedded if and only if
\begin{equation}\label{eq2}
 \sqrt{\frac{4\mu_1\mu_2\mu_3}{(\mu_1\mu_2+\mu_1\mu_3+\mu_2\mu_3)^2}-1}\geqslant \frac{\pi}{-\log\abs{\lambda}},\mbox{\, when\, } \frac{1}{\mu_1},\frac{1}{\mu_2},\frac{1}{\mu_3} \mbox{\,satisfy the triangle inequality, }
\end{equation}
or
\begin{equation}
   \hspace{-9mm} \sqrt{ \frac{m}{1-m}}\geqslant \frac{\pi}{-\log\abs{\lambda}},\mbox{\, otherwise. }
\end{equation}
\end{thm}
\noindent{\it Proof of Theorem \ref{main}.\,}
Clearly, if the function $H(\kappa,\gamma,\delta,\nu)$ reaches its maximum then $\nu=\min\set{\kappa,\gamma,\delta}$. From Proposition~\ref{prop3} we have to find a transition rate matrix $Q(\kappa, \gamma, \delta, \nu)$ of the form just above (\ref{parameter}), with eigenvalues $\theta(\kappa, \gamma, \delta, \nu)$ and $\bar{\theta}(\kappa, \gamma, \delta, \nu)$ for
which $e^{\theta h} = \lambda$ for some $h > 0$.
Using Proposition~\ref{lem5} we can find $(\kappa_0, \gamma_0, \delta_0, \nu_0)$, so that for the given $\mu',\,\lambda$ we
have
\begin{equation}
  H(\kappa_0, \gamma_0, \delta_0, \nu_0)= \frac{\pi}{-\log\abs{\lambda}}=H_0.
\end{equation}
The corresponding transition rate matrix $Q(\kappa_0, \gamma_0, \delta_0, \nu_0)$ has eigenvalue $\theta(\kappa_0, \gamma_0, \delta_0, \nu_0)=\theta_0$
and $\bar{\theta}_0$, where
\begin{equation}
  \theta_0=-\frac{\alpha_0}{2}+\mi \sqrt{\beta_0-\alpha_0^2/4}=\frac{\alpha_0}{2}[-1+\mi H_0]=\frac{\alpha_0}{2}[-1+\mi \frac{\pi}{-\log\abs{\lambda}}].
\end{equation}
We finally choose $h =-2 \log\abs{\lambda} /\alpha_0$ and find
\begin{eqnarray*}
   \theta_0 h &=-\frac{2\log\abs{\lambda}}{\alpha_0} \frac{\alpha_0}{2}[-1+\mi \frac{\pi}{-\log\abs{\lambda}}]\\
              &= \log\abs{\lambda}+\mi \pi
\end{eqnarray*}
which satisfies $e^{\theta_0 h} = \lambda $.
{\hfill\large{$\Box$}}
\begin{rem}
If $\mu_1+\mu_2+\mu_3=1$, and $\frac{1}{\mu_1},\,\frac{1}{\mu_2},\,\frac{1}{\mu_3}$ satisfy the triangle inequality, then it can be shown easily that
\begin{equation}
  4\mu_1\mu_2\mu_3\geqslant(\mu_1\mu_2+\mu_1\mu_3+\mu_2\mu_3)^2\geqslant 3\mu_1\mu_2\mu_3,
\end{equation}
i.e., the term inside the $\sqrt{\cdot}$ of Eq.(\ref{eq2}) is positive.
\end{rem}
\begin{rem}\label{rem}
For a fixed $\tensor{P}^{\infty}=\vec{e}\mu' $, what are the possible values of $\lambda<0$ that make the stochastic matrix $\tensor{P}$ embeddable?
In \cite[Theorem 3.7]{cart}, P. Carette puts the above question and shows that there exists $\Lambda<0$ such that $\tensor{P}$ is imbeddable
if and only if $\Lambda\leqslant \lambda <0$. As a consequence of Theorem~\ref{main},  the exact value of $\Lambda$ is that
\begin{equation}
\hspace{-4mm}  \Lambda=\left\{
      \begin{array}{ll}
      -\exp\set{-\frac{\pi}{\sqrt{b}}}, &\mbox{when\,} \frac{1}{\mu_1},\frac{1}{\mu_2},\frac{1}{\mu_3} \mbox{\,satisfy the triangle inequality,\, }\\
      -\exp\set{-\sqrt{\frac{1-m}{m}}\,\pi}, &\mbox{otherwise},
      \end{array}
\right.
\end{equation}
where $b=\frac{4\mu_1\mu_2\mu_3}{(\mu_1\mu_2+\mu_1\mu_3+\mu_2\mu_3)^2}-1,\, m=\min\set{\mu_1,\,\mu_2,\,\mu_3}$.
\end{rem}
\subsection{Proof of the propositions}\label{prf}
\begin{lem}\label{cor0}
  If $\tensor{P}$ satisfies Eq.(1), then its right eigenvectors with eigenvalue $\lambda$ span the orthogonal complement of $\mu$.
\end{lem}
\begin{proof}
  If $f$ is a right eigenvector of $\tensor{P}$ with eigenvalue $\lambda$ then
\begin{equation}
  \lambda f= \tensor{P} f = \vec{e} \mu' f +\lambda (\tensor{I}-\vec{e} \mu') f =(1-\lambda)\vec{e} \mu' f + \lambda f,
\end{equation}
so that $\mu' f=0$. Hence the eigenvectors span the orthogonal complement of $\mu$.
\end{proof}
\noindent{\it Proof of Proposition~\ref{prop3}.\,}
The necessity.  Since $\lambda<0$ and $\tensor{P}=e^{h\tensor{Q}}$, $\tensor{Q}$ has complex eigenvalues $\theta$ and $\bar{\theta}$ such that $e^{\theta h}=\lambda$. It follows from Lemma~\ref{lem0} that $\tensor{Q}$ has $\mu' \tensor{Q}=0$.

The sufficiency. Since $\tensor{P}$ satisfies Eq.(\ref{PP}), one obtains that
\begin{equation}\label{f-mat}
  \tensor{P}=\tensor{F}\,\diag\set{1,\,\lambda,\,\lambda}\tensor{F}^{-1},
\end{equation}
where $\tensor{F}=[\vec{e},\,\varphi_1,\,\varphi_2]$, $\varphi_1,\,\varphi_2$ are any two linear independent vectors in the orthogonal complement of $\mu$ by Lemma~\ref{cor0}.

Denote by $f+\mi g$ the eigenvector of $\tensor{Q}$ with eigenvalue $\theta=p+\mi q,\,q\neq 0$. Clearly
$f,\,g$ are linear independent. $\mu'\tensor{Q}=0$ implies that $f,\,g$ span the orthogonal complement of $\mu$.

Let $\tensor{P}_h=e^{h\tensor{Q}}$. Hence $e^{h\tensor{Q}}=\sum_{n=0}^{\infty}\frac{h^n}{n!}\tensor{Q}^n$ implies
that if $\mu'\tensor{Q}=0 $ then $\mu'\tensor{P}_h=\mu'$ and that
\begin{equation}
  \tensor{P}_h(f+\mi g)=e^{\theta h}(f+\mi g)={\lambda}(f+\mi g).
\end{equation}
Hence $\tensor{P}_h f={\lambda}f,\, \tensor{P}_h g={\lambda}g$. Any transition rate matrix $\tensor{Q}$ has $\tensor{Q}\vec{e}=0$, so that
$\tensor{P}_h \vec{e}=\vec{e}$.
Thus one obtains that
\begin{equation}\label{f2-mat}
  \tensor{P}_h=[\vec{e},\,f,\,g]\,\diag\set{1,\,\lambda,\,\lambda}[\vec{e},\,f,\,g]^{-1}.
\end{equation}

Note that in Eq.(\ref{f-mat}), one can choose that the matrix $\tensor{F}=[\vec{e},\,f,\,g]$, which implies that $\tensor{P}_h=\tensor{P}$.
{\hfill\large{$\Box$}}

\begin{lem}\label{lemchen}
  Suppose that $f(x)=x+\frac{a}{x}$ with $0<x\leqslant c$, where $a,\,c>0$ are two constants. Then the minimum of $f(x)$ is
\begin{equation}\label{minm}
  \left\{
      \begin{array}{ll}
      2\sqrt{a}, &\mbox{\,when \,} \sqrt{a}\leqslant c ,\\
      c+\frac{a}{c}, &\mbox{\, when\, } \sqrt{a}> c.
      \end{array}
\right.
\end{equation}
The minimum is attained at the point
\begin{equation}\label{min1}
  \left\{
      \begin{array}{ll}
      x=\sqrt{a}, &\mbox{\, when\,  } \sqrt{a}\leqslant c ,\\
      x=c,  &\mbox{\, when\, } \sqrt{a}> c.
      \end{array}
\right.
\end{equation}
\end{lem}
\begin{proof}
  It is trivial.
\end{proof}
Denote that
\begin{eqnarray}
  \mathcal{D}=\set{(x_1,x_2,x_3)\in \Rnum^3:\, x_1,\,x_2,\,x_3\geqslant0, x_1+x_2,\,x_2+x_3,\,x_3+x_1 >0},\\
  \mathcal{E}_1=\mathcal{D} \cap \set{(x_1,x_2,x_3)\in \Rnum^3:\, x_1\leqslant x_2\leqslant x_3}.
\end{eqnarray}
 Since $F(x_1,x_2,x_3)=F(rx_1,rx_2,rx_3),\,\forall r>0$, the
 existence of the maximum of $F(x_1,x_2,x_3)$ on $\mathcal{D}$ is equal to the existence on the unit sphere which
 holds true since the unit sphere is compact.\footnote{The points $(1,0,0),\,(0,1,0),\,(0,0,1)$ are not in $\mathcal{D}$,
 but the values of $F(x_1,x_2,x_3)$ are zero at these points, which do not alter the maximum of $F(x_1,x_2,x_3)$.}
\begin{lem}\label{lem11}
  Restricted on $\mathcal{E}_1$, the maximum of $F(x_1,x_2,x_3)$ is
\begin{equation}\label{maximum1}
  \left\{
      \begin{array}{ll}
      \frac{\mu_1\mu_2\mu_3}{4(1-\mu_1 )}, &\mbox{\,when \,} \frac{1}{\mu_1}\geqslant \frac{1}{\mu_2}+\frac{1}{\mu_3} ,\\
      \frac{1}{(\frac{1}{\mu_1}+\frac{1}{\mu_2}+\frac{1}{\mu_3})^2}, &\mbox{\, when\, } \frac{1}{\mu_1}<\frac{1}{\mu_2}+\frac{1}{\mu_3}.
      \end{array}
\right.
\end{equation}
The maximum is attained at the points
\begin{equation}\label{maximumpoint1}
  \left\{
      \begin{array}{ll}
      x_1=x_2=\frac{\frac{1}{\mu_2}+\frac{1}{\mu_3}}{\frac{2}{\mu_1}-\frac{1}{\mu_2}-\frac{1}{\mu_3}} x_3, &\mbox{\, when\,  } \frac{1}{\mu_1}\geqslant\frac{1}{\mu_2}+\frac{1}{\mu_3} ,\\
      x_1=x_2=x_3,  &\mbox{\, when\, } \frac{1}{\mu_1}<\frac{1}{\mu_2}+\frac{1}{\mu_3}.
      \end{array}
\right.
\end{equation}
\end{lem}
\begin{proof}
  Restricted on $\mathcal{E}_1$, i.e., $x_1\leqslant x_2\leqslant x_3$, we obtain that
\begin{equation}
  F(x_1,x_2,x_3)=\frac{x_1x_2+x_2x_3+x_3x_1+x_1^2}{(\frac{x_1+x_2}{\mu_1}+
\frac{x_2+x_3}{\mu_2}+\frac{x_3+x_1}{\mu_3})^2}=\frac{(x_1+x_2)(x_3+x_1)}{(\frac{x_1+x_2}{\mu_1}+
\frac{x_2+x_3}{\mu_2}+\frac{x_3+x_1}{\mu_3})^2}.
\end{equation}
Let $r=x_1+x_2,\,s=x_2+x_3,\,t=x_3+x_1$. Then $0<r\leqslant t\leqslant s,$ and
\begin{equation}
  F(x_1,x_2,x_3)=\frac{rt}{(\frac{r}{\mu_1}+
\frac{s}{\mu_2}+\frac{t}{\mu_3})^2}\triangleq G(r,s,t)\leqslant G(r,t,t).
\end{equation}
That is to say, $G(r,s,t)$ attains its maximum when $s=t$ since it is a decreasing function of $s$ .
Let $w=\sqrt{\frac{r}{t}}$. Then $0<w\leqslant 1$ and we have that
\begin{eqnarray}
  G(r,t,t)&=&\frac{rt}{[\frac{1}{\mu_1}r+
(\frac{1}{\mu_2}+\frac{1}{\mu_3})t]^2}\nonumber \\
&=&\big[\,\frac{\mu_1}{w+\mu_1(\frac{1}{\mu_2}+\frac{1}{\mu_3})\frac{1}{w}}\,\big]^2 \triangleq L(w).
\end{eqnarray}
It follows from Lemma~\ref{lemchen} that when $\mu_1(\frac{1}{\mu_2}+\frac{1}{\mu_3})>1$, the maximum of $L(w)$ is
$\frac{1}{(\frac{1}{\mu_1}+\frac{1}{\mu_2}+\frac{1}{\mu_3})^2}$, and when $\mu_1(\frac{1}{\mu_2}+\frac{1}{\mu_3})\leqslant1$, the maximum of $L(w)$ is $\frac{\mu_1}{4(\frac{1}{\mu_2}+\frac{1}{\mu_3})}$.
In addition, the direct computation yields the maximum points (\ref{maximumpoint1}) from Lemma~\ref{lemchen}.
\end{proof}

\noindent{\it Proof of Proposition~\ref{lem5}.\,}
In order to prove Proposition 2.5 we first assume that $\mu_i$ satisfies the
triangular inequality. There is no loss of generality in assuming that $x_1 \leqslant x_2 \leqslant x_3$ because if this is not the case, the arguments can be permuted to satisfy the restriction. This will leave the value of $F$ invariant if also $\mu_i$ are
permuted correspondingly, but the triangular inequality condition is invariant
to permutations of $\mu_i$. Thus from Lemma~\ref{lem11} we get that $F$ is bounded by
the restricted maximum:
\begin{equation}
  F(x_1, x_2, x_3)\leqslant \frac{1}{(\frac{1}{\mu_1}+\frac{1}{\mu_2}+\frac{1}{\mu_3})^2}.
\end{equation}
If, however, $\mu_i$ does not satisfy the triangular inequality, then, without
loss of generality we can assume that
\begin{equation}
  \frac{1}{\mu_1}\geqslant \frac{1}{\mu_2}+\frac{1}{\mu_3}.
\end{equation}
Now on the set $\mathcal{E}_1$ we can apply Lemma~\ref{lem11} and find that
the function is bounded by the unrestricted maximum
\begin{equation}
  F(x_1, x_2, x_3)\leqslant \frac{\mu_1\mu_2\mu_3}{4(1-\mu_1 )}.
\end{equation}
Now when $\frac{1}{\mu_1}\geqslant \frac{1}{\mu_2}+\frac{1}{\mu_3}$, we have $\mu_1\leqslant \frac{\mu_2\mu_3}{\mu_2+\mu_3}\leqslant\min(\mu_2,\,\mu_3)$, so that $\mu_1=m=\min(\mu_1,\,\mu_2,\,\mu_3)$ and
\begin{equation}
  F(x_1, x_2, x_3)\leqslant \frac{\mu_1\mu_2\mu_3}{4(1-\mu_1 )}\leqslant\frac{\mu_1\mu_2\mu_3}{4(1-m )}.
\end{equation}

{\hfill\large{$\Box$}}

\section{The imbedding problem for 3-order transition matrix with positive eigenvalues or complex eigenvalues}

\begin{prop}
Suppose $\tensor{P}$ be an indecomposable $3\times 3$ transition probability matrix with eigenvalues $\set{1,\,\lambda_1,\,\lambda_2}$. Then it satisfies that
\begin{equation}\label{eigen2}
         \tensor{P}^2-(\lambda_1+\lambda_2)\tensor{P}+\lambda_1\lambda_2\tensor{I}=(\lambda_1-1)(\lambda_2-1)\tensor{P}^{\infty}.
\end{equation}
\end{prop}
It is exactly the Eq.(1.15) of \cite{jhson}.
Johansen S. points out that the condition of the imbedding problem can be given in terms of $\tensor{P}$ and $\tensor{P}^{\infty}$ (or $\mu$) in \cite{jhson}. The following two propositions are the direct corollary of Proposition~1.2, Proposition~1.4 of \cite{jhson} and Eq.(\ref{eigen2}).
\begin{prop}
  Let $\tensor{P}$ be an indecomposable $3\times 3$ transition probability matrix with positive eigenvalues $\set{1,\,\lambda_1,\,\lambda_2}$. $\tensor{P}$ can be imbedded if and only if
    \begin{equation}\label{positive}
      {p_{ij}} \geqslant {\mu_j} \frac{(\lambda_2-1)\log \lambda_1-(\lambda_1-1)\log \lambda_2}{\log \lambda_2-\log \lambda_1},\quad i\neq j.
    \end{equation}
 If $\lambda_1=\lambda_2=\lambda$, then the right hand side of (\ref{positive}) is in the sense of limit, i.e., fix the value of $\lambda_1$, and let $\lambda_2\rightarrow \lambda_1$, one has that
    \begin{equation}
      {p_{ij}} \geqslant {\mu_j}( \lambda\log \lambda-\lambda +1 ),\quad i\neq j.
    \end{equation}
 \end{prop}

\begin{prop}
  Let $\tensor{P}$ be an indecomposable $3\times 3$ transition probability matrix with complex eigenvalues $\set{1,\,\lambda_1,\,\lambda_2}$, $\lambda_1=r e^{\mi \theta}\,,\lambda_2=r e^{-\mi \theta},\,\theta\in (0,\,\pi)$. $\tensor{P}$ can be imbedded if and only if
    \begin{equation}
     {p_{ij}} \geqslant{\mu_j}( 1-r\cos \theta +\frac{\sin\theta}{\theta}\,r\log r),\quad i\neq j.
    \end{equation}
or
    \begin{equation}
      {p_{ij}} \leqslant{\mu_j}( 1-r\cos \theta +\frac{\sin\theta}{2\pi-\theta}\,r\log r),\quad i\neq j.
    \end{equation}
 \end{prop}
The two propositions appear implicitly in \cite{jhson} and seem neater than Eq.(1.11-1.13,1.17) of Reference \cite{jhson}.
\begin{rem}
 If $\tensor{P}$ is reversible, then $\frac{p_{ij}}{\mu_j} =\frac{p_{ji}}{\mu_i} $. One should only test half number of the inequalities.
\end{rem}
\section{Conclusion}
In the present paper, we solve the imbedding problem for 3-order transition matrix with coinciding negative eigenvalues by means of an alternate parameterization of the transition rate matrix, which is different from the traditional way to calculate the matrix logarithm or the matrix square root.
\section*{Acknowledgements}  Many thanks to the
anonymous referee for the helpful comments and suggestions
leading to the improvement of the paper. Thanks to Jianmin Chen, one of my undergraduate students, for showing me the proof of the key Lemma~\ref{lem11}.
This work is supported by Hunan Provincial Natural Science Foundation of China (No 10JJ6014).



\end{document}